\documentclass[11pt,a4paper,twoside]{amsart}
\usepackage[english]{babel}
\usepackage{latexsym,layout,verbatim,amssymb,amsmath,setspace,cite}
\usepackage{amsthm,mathrsfs}
\usepackage{cleveref,enumerate,fontenc}
\usepackage[utf8]{inputenc}
\usepackage[T1]{fontenc}
\newtheorem{theorem}{Theorem}[section]

\newtheorem{corollary}[theorem]{Corollary}

\newtheorem{definition}[theorem]{Definition}
\newtheorem{lemma}[theorem]{Lemma}
\newtheorem{proposition}[theorem]{Proposition}
\newtheorem{remark}[theorem]{Remark}

\newtheorem{example}[theorem]{Example}

\newcommand{\enne}{\mathbb{N}}

\begin{document}

\title[Atomicity related to non-additive integrability]{Atomicity related to non-additive integrability}

\author[Candeloro]{\bf Domenico Candeloro}
\address{Department of Mathematics and Computer Sciences \\ University of Perugia\\
Via Vanvitelli, 1 - 06123 Perugia (Italy)\\ (orcid id 0000-0003-0526-5334)}
\email{candelor@dmi.unipg.it}

\author[Croitoru]{\bf Anca Croitoru}
\address{\rm Faculty of Mathematics,\\ ''Al.I. Cuza'' University\\ Bd. Carol I, no 11, Ia\c{s}i, 700506, 
(Rom\^{a}nia)}
\email{croitoru@uaic.ro}

\author[Gavrilu\c{t}]{\bf Alina Gavrilu\c{t}}
\address{\rm Faculty of Mathematics,\\ ''Al.I. Cuza'' University \\ Bd. Carol I, no 11, Ia\c{s}i, 700506, 
(Rom\^{a}nia)}
\email{gavrilut@uaic.ro}

\author[Sambucini]{\bf Anna  Rita Sambucini}
\address{\rm Department of Mathematics and Computer Sciences \\ University of Perugia\\
Via Vanvitelli, 1 - 06123 Perugia (Italy) \\  (orcid id 0000-0003-0161-8729)}
\email{anna.sambucini@unipg.it}



\begin{abstract} 
In this paper we present some results concerning Gould integrability of
vector functions with respect to a monotone measure on finitely purely
atomic measure spaces. As an application a   Radon-Nikodym theorem in this setting is obtained.
\end{abstract}
\maketitle


\section{Introduction}\label{intro}
In the last years, the field of non-additive measures was intensively 
used  in a wide range of areas such as economics, social sciences,
 biology and philosophy and it gives a mathematical framework for 
describing a situation of conflict or cooperation between intelligent rational players, in order to 
 predict the outcome of a  process.
 In this field, the theory of (pseudo)atoms and monotonicity is used in statistics, game theory, probabilities, artificial intelligence.
\\
Purely atomic measures were studied in literature (in different variants)
due to their special form and their special properties. In this case, the
entire space is assumed to be a finite collection of pairwise disjoint atoms
and an atom can be viewed as a black hole. For instance, Chi\c{t}escu  \cite{c1975,c2001}
and Leung  \cite{l1997} established some relationships with classical problems in $
L^{p}$ spaces, Iona\c{s}cu and Stancu  \cite{is2010} have obtained concrete
independent events in purely atomic probability spaces with geometric
distribution, Matveev  \cite{m2010} has proved that every $\sigma $-finite Borel
measure defined on a Rothberger space is purely atomic, Elton and Hill \cite{eh}
studied the ham sandwich theorem.
\\
The subject of our paper belongs to the non-additivity domain that was
intensively studied by many authors for interesting and important properties (e.g.  \cite{klmp2015,lmp2014,lmp,ms1994,m2008,nm2004,p1994,bcs2015,cs2014,cs2015a,ccgs2015,ccgs2015b,cdms2016}).
In this paper we present some results of Gould integrability  \cite{g1965} on
finitely purely atomic measure spaces. 
The idea is similar to that of \cite{gcmg2009}, where mainly the measures are set-valued and the functions are scalar, 
while in the present research the other product by scalar  is considered, namely vector valued functions and real valued
 and positive measures are discussed.
Convergence and Radon-Nikod\'ym theorems are also obtained in this framework.

The structure of the paper is the following: in Section 2 we give some
preliminaries. Section 3 contains some results on atoms that we shall use in
the sequel; examples are given  to show that some of them do not hold in general in Section 4, 
together with different properties regarding Gould
type integrability on finitely purely atomic measure spaces, such as: a
Lebesgue type theorem of convergence and comparative 
 results among Gould integrability, Choquet integrability, total measurability and boundedness.
Finally, in Section 5 we establish a Radon-Nikod\'ym theorem for purely atomic measures.
\section{Preliminaries}\label{two}

$T$ is an abstract nonvoid set, $\mathcal{P}(T)$ the family of all subsets
of $T$, $\mathcal{A}$ an algebra of subsets of $T$ and $m:\mathcal{A}
\rightarrow [0,+\infty)$ an arbitrary set function, with $m(\emptyset)=0$.
If $A\subseteq T$, then $T\setminus A$ will be denoted by $A^{c}.$ 
\\
We recall some notions that will be used throughout the paper.

\begin{definition}\label{def1} \rm 
( \cite{d1974,d1972,p1995}) $m$ is said to be: 
\begin{itemize}
\item\textit{a monotone measure} if $m(A)\leq m(B),$ for every $A,B\in \mathcal{A}$, with $A\subseteq B$;
\item 
\textit{\ null-additive} if $m(A\cup B)=m(A)$, for every $A,B\in
\mathcal{A}$, with $m(B)=0$; 
\item
\textit{$\sigma$-null-additive} if $m(\bigcup_n A_n)=0$ as soon as $A_n\in \mathcal{A}$ and $m(A_n)=0$ for all $n$.
\item 
 \textit{subadditive} if $m(A\cup B)\leq m(A) + m(B)$, for every $A,B\in\mathcal{A}$; 
\item 
\textit{finitely additive} if $m(A\cup B) = m(A) + m(B)$, for
every $A,B\in\mathcal{A}$, with \mbox{$A\cap B = \emptyset$;} 
\item
 \textit{$\sigma$-subadditive} if $m(\bigcup\limits_{n=1}^{\infty}A_{n}) \leq 
\sum\limits_{n=1}^{\infty}m(A_{n})$, for every 
$\{A_{n}\}_{n\in \mathbb{N}}\subset \mathcal{A}$, so that $
\bigcup\limits_{n=0}^{\infty}A_{n}\in\mathcal{A}$; 
\end{itemize}
\end{definition}

\begin{remark}\label{rem1} \rm
 If $m$ is monotone and subadditive, then $m$ is null-additive. 
 A subadditive monotone measure is sometimes called a submeasure
(\cite{d1974}). 
\end{remark}

\begin{definition}\label{def-v} \rm
Let $m:\mathcal{A} \to [0,+\infty)$ with $m(\emptyset) = 0.$ 
\begin{itemize}
\item[\ref{def-v}.i)] The set function $\overline{m}:\mathcal{P}(T)\to [0,+\infty],$
called \textit{the variation} of $m$, is defined by $\overline{m}(E)=\sup
\{\sum\limits_{i=1}^{n}m(A_{i})\}$, for every $E\in \mathcal{P}(T)$, where
the supremum is extended over all finite families of pairwise disjoint sets $
\{A_{i}\}_{i=1}^{n}\subset \mathcal{A}$ with $A_{i}\subseteq E$, for every $
i\in \{1,\ldots,n\}$. 
\item[\ref{def-v}.ii)] $m$ is said to be \textit{of finite variation on }$\mathcal{A}$
if $\overline{m}(A)<\infty $, for every $A\in \mathcal{A}$ (or,
equivalently, if $\overline{m}(T)<\infty$).
\item[\ref{def-v}.iii)] We consider the set function $\widetilde{m}$ defined by $
\widetilde{m}(E)=\inf \{\overline{m}(A);E\subseteq A,A\in \mathcal{A}\},\;%
\text{for every}\;E\in \mathcal{P}(T).$ 
\end{itemize}
\end{definition}

The statements in the following Remark easily follow by the definitions: some of them are contained 
in \cite[Remark 2.4]{gcmg2009}, though in a more abstract situation.
\begin{remark}\label{rem-bar} \rm \mbox{~} 
\begin{itemize}
\item[\ref{rem-bar}.i)] If $E\in\mathcal{A}$, then in the definition of $\overline{m}$
the supremum could be considered over all finite families of pairwise
disjoint sets $\{A_{i}\}_{i=1}^{n}\subset \mathcal{A}$, so that $
\bigcup\limits_{i=1}^{n}A_{i} = E.$ 
\item[\ref{rem-bar}.ii)] $m(A)\leq \overline{m}(A)$, $\forall A\in\mathcal{A}$. So, if 
$\overline{m}(A) = 0$, then $m(A)=0,$ $\forall A\in\mathcal{A}$. 
\item [\ref{rem-bar}.iii)] Suppose $m$ is monotone and let $A\in\mathcal{A}$ with $m(A) =
0.$ Then $\overline{m}(A) = 0.$ In consequence, if $m$ is monotone, then $
\overline{m}(A) = 0 \Leftrightarrow m(A) = 0$, for every $A\in \mathcal{A}$.
\item [\ref{rem-bar}.iv)] $\overline{m}$ and $\widetilde{m}$ are monotone. 
\item [\ref{rem-bar}.v)] If $m$ is subadditive ($\sigma$-subadditive respectively), then 
$\overline{m}$ is additive ($\sigma$-additive respectively) on $\mathcal{A}$.
\item[\ref{rem-bar}.vi)] If $m$ is finitely additive, then $\overline{m}(E) = \sup
\{m(B)| B \in \mathcal{A}, B\subseteq E\}$, for every $E\in \mathcal{P}(T)$
and $\overline{m}(A) = m(A)$ for every $A\in \mathcal{A}$. 
\item[\ref{rem-bar}.vii)] $\widetilde{m}(A)=\overline{m}(A)$, for every $A\in \mathcal{A}$. 
\end{itemize}
\end{remark}

In what follows, we give some results regarding null-additivity of $\overline{m}$ and $\widetilde{m}$.

\begin{proposition}
If $m$ is null-additive, then $\overline{m}$ is null-additive on $\mathcal{A}$.
\end{proposition}
\begin{proof} 
Let $A,B\in \mathcal{A}$ with $\overline{m}(B) = 0.$ Since $\overline{m}$ is monotone,
we only have to prove that $\overline{m}(A\cup B)\leq \overline{m}(A)$. 
Let $\{C_{i}\}_{i=1}^{p}\subset \mathcal{A}$, $C_{i}\cap C_{j} = \emptyset$, $i\neq j$ so that 
$\bigcup\limits_{i=1}^{p}C_{i} = A\cup B$. We have two situations.
\\
Suppose that $A\cap B = \emptyset$. Then $B = \bigcup\limits_{i=1}^{p}(C_{i}\cap A^{c})$. 
From $\overline{m}(B) = 0$ we obtain $m(B) = 0.$ So $m(C_{i}\cap A^{c}) = 0,$ 
$\forall i \in \{1,\ldots,p\}.$ From the null-additivity of $m$, it results $m(C_{i}) = m(C_{i}\cap A)$, 
for every $i\in \{1,\ldots,p\}.$ So, $\sum\limits_{i=1}^{p}m(C_{i})=\sum\limits_{i=1}^{p}m(C_{i}\cap A) 
\leq \overline{m}(A)$ which implies that $\overline{m}(A\cup B) \leq \overline{m}(A)$.\\
If $A\cap B \neq \emptyset$ then we have $A\cup B = (A\setminus B) \cup B$ and 
$\overline{m}(A\cup B) = \overline{m}((A\setminus B) \cup B )= \overline{m}(A\setminus B) \leq \overline{m}(A).$
\end{proof}

\begin{proposition}
Suppose $\mathcal{A}$ is a $\sigma$-algebra. If $\overline{m}$ is null-additive, 
then $\widetilde{m}$ is null-additive on $\mathcal{P}(T)$.
\end{proposition}
\begin{proof}
 Let $A,B\in \mathcal{P}(T)$ with $\widetilde{m}(B)=0.$ Since $\widetilde{m}$ is monotone, 
we only have to prove that $\widetilde{m}(A\cup B) \leq \widetilde{m}(A)$. Because $\widetilde{m}(B) = 0,$ 
for every $n\in \mathbb{N}$ there exists $B_{n}\in\mathcal{A}$, $B\subseteq B_{n}$ so that
 $\overline{m}(B_{n})<n^{-1}$. It results $\lim\limits_{n\to \infty}\overline{m}(B_{n}) = 0.$ 
Without any loss of generality, we may suppose that  $B_{n}\searrow B^{\prime}$. 
So, $B\subseteq B^{\prime}$ and $B^{\prime}\in\mathcal{A}$. Since $\overline{m}$ is  
monotone, it follows that $\overline{m}(B^{\prime})\leq\overline{m}(B_n)$ for all $n$,
 hence $\overline{m}(B^{\prime})=0.$ Now, let $C\in\mathcal{A}$ be arbitrary with $A\subseteq C.$ 
Then $\widetilde{m}(A\cup B) \leq \overline{m}(B^{\prime}\cup C) = \overline{m}(C)$, 
whence $\widetilde{m}(A\cup B) \leq \widetilde{m}(A)$, which finishes the proof.
\end{proof}

\section{Atoms}\label{three}
In this section some properties about atomicity and total measurability are investigated. We recall that
\begin{definition}\label{d-atom} \rm 
\begin{itemize}
\item  A set $A\in $ $\mathcal{A}$ is said to be an \textit{atom} of $m$
if $m(A)>0$ and for every $B\in $ $\mathcal{A}$, with $B\subset A$, we have
 $m(B)=0$ or $m(A\setminus B)=0$. 
\item For further reference, atoms can be introduced also for vector-valued measures: if $m$ takes values in a Banach space $X$, 
a {\em null set} is an element $N\in \mathcal{A}$ such that $m(B)=0$ for every measurable $B\subset N$; and an 
{\em atom} for $m$ is any non-null set $A\in \mathcal{A}$ such that, for every measurable $B\subset A$ 
one has that $B$ is null or $A\setminus B$ is null.
\item  $m$ is said to be \textit{finitely purely atomic} (and  $T$ 
\textit{a finitely purely atomic space}) if there is a finite
disjoint family $\{A_{i}\}_{i=1}^{n}\subset \mathcal{A}$ of atoms of $m$ so
that $T=$ $\underset{i=1}{\overset{n}{\cup }}A_{i}$.
\end{itemize}
\end{definition}
\begin{lemma}\label{1di3} {\rm (\cite[Remark 3.7]{gcmg2009})}
Let $m:\mathcal{A}\to [0,+\infty)$ be a non-negative set function, with $
m(\emptyset) = 0$ and let $A\in \mathcal{A}$ be an atom of $m$.
\begin{itemize}
\item[\rm \ref{1di3}.1)] If $m$ is monotone and the set $B\in \mathcal{A}$ is so that $B\subseteq
A $ and $m(B) >0$, then $B$ is also an atom of $m$ and $m(A \setminus B) =
0.$ Moreover, if $m$ is null-additive, then $m(B) = m(A)$.
\item[\rm \ref{1di3}.2)]
If $m$ is monotone and null-additive, then for every finite partition $
\{B_{i}\}_{i=1}^{n}$ of $A$, there exists only one $i_{0}\in
\{1,2,\ldots,n\} $ so that $m(B_{i_{0}}) = m(A)$ and $m(B_{i}) = 0$ for
every $i\in \{1,\ldots,n\}$, $i\neq i_{0}$.
\end{itemize}
\end{lemma}
\begin{proof}
It is enough to assume that the values of the multimeasure in   \cite[Remark 3.7]{gcmg2009} are (real) singletons.
\end{proof}

\begin{example}\label{es1} \rm \mbox{}
\begin{itemize}
\item[\ref{es1}.i)] 
 Let $T = \{a,b,c\},$ $\mathcal{A} = \mathcal{P}(T)$ and
\[  m(A) =
\begin{cases}
2, & \text{if } A=T \\
1, & \text{if } A = \{a,b\} \text{ or } A= \{c\} \\
0, & \text{otherwise}.
\end{cases}
\]
We remark that $A_{1} = \{a,b\}$ and $A_{2} = \{c\}$ are disjoint
atoms of $m$ and $T = A_{1}\cup A_{2}$. So $m$ is finitely purely atomic. 
\item[\ref{es1}.ii)] 
 Let $T$ be a countable set, $\mathcal{A} = \{A\subseteq T| A$ is
finite or $A^{c}$ is finite$\}$ and $m:\mathcal{A}\to [0,\infty)$ defined
for every $A\in \mathcal{A}$ by
\[ m(A) =
\begin{cases}
0, & \text{if } A \text{ is finite } \\
1, & \text{if } A^{c} \text{ is finite}.
\end{cases} \]
\end{itemize}
Then every set $A\in \mathcal{A}$, such that $A^{c}$ is finite, is an atom
of $m$.
\end{example}

\begin{remark}\label{rem3}\rm (\cite[Remark 3.8]{gcmg2009})
 If $m$ is monotone, then the following statements are equivalent:
\begin{itemize}
\item  $m$ is finitely purely atomic $\Longleftrightarrow$
  $\overline{m}$ is finitely purely atomic on $\mathcal{A}
\Longleftrightarrow \widetilde{m}$ is finitely purely atomic on $\mathcal{A}.$
\item
 If $m$ is a null-additive monotone measure and $A\in \mathcal{A}$
is an atom of $m$, then $\overline{m}(A)=m(A).$\\ Indeed, let $
\{B_{i}\}_{i=1}^{n}$ be a partition of $A$. According to Lemma \ref{1di3}.2), there
is a unique $i_{0}\in \{1,\ldots ,n\}$ such that $m(B_{i_{0}})=m(A)$ and $
m(B_{i})=0$ for every $i\in \{1,\ldots ,n\}$, $i\neq i_{0}$. So, we have $
\sum\limits_{i=1}^{n}m(B_{i})=m(A)$, which implies that $\overline{m}
(A)=m(A) $.
\item
 If $m$ is a finitely purely atomic subadditive monotone
measure, then $m$ is of finite variation.\\ Indeed, suppose $T =
\bigcup\limits_{i=1}^{p}A_{i}$, where $\{A_{i}\}_{i=1}^{p}\subset \mathcal{A}
$ are pointwise disjoint atoms of $m$. Then 
\[\overline{m}(T) = \overline{m}
(\bigcup\limits_{i=1}^{p}A_{i}) = \sum\limits_{i=1}^{p}\overline{m}(A_{i}) =
\sum\limits_{i=1}^{p}m(A_{i})<\infty.\] 
\end{itemize}
\end{remark}

In order to state our next theorems,  a result of \cite{lmp} will be presented.
In the sequel let $T$ be a locally compact Hausdorff topological space, $
\mathcal{K}$ be the lattice of all compact subsets of $T$, $\mathcal{B}$ be
the Borel $\sigma$-algebra (that is the smallest $\sigma$-algebra containing
$\mathcal{K}$) and $\tau$ be the class of all open sets. For a study on this subject 
it is possible to see also \cite{p1995a}.

\begin{definition}\rm
A set function $m:\mathcal{B}\rightarrow \lbrack 0,+\infty )$
is called \textit{regular} if for each set $A\in \mathcal{B}$ and each $
\varepsilon >0,$ there exist $K\in \mathcal{K}$ and $D\in \tau $ such that $
K\subseteq A\subseteq D$ and $m(D\setminus K)<\varepsilon .$ 
\end{definition}

\begin{theorem}\label{cor-at}
{\rm \cite[Theorem 4.6]{lmp}} Let $m:\mathcal{B}\rightarrow \lbrack 0,+\infty )$ be a
regular null-additive monotone set function. If $A\in \mathcal{B}$ is an
atom of $m$, then there exists a unique point $a\in A$ such that $
m(A)=m(\{a\})$ and  $m(A\setminus \{a\})=0.$
\end{theorem}

\begin{proposition} 
Suppose $m:\mathcal{B}\to [0,+\infty)$ is finitely purely atomic,
regular, null-additive and monotone. 
Then there exists a finite
family $\{A_{i}\}_{i=1}^{n}\subset \mathcal{A}$ of pairwise disjoint atoms
of $m$ so that $T = \bigcup_{i=1}^{n} A_{i}$. 
\end{proposition}
\begin{proof}
By Theorem \ref{cor-at}, there
are unique $a_{1}, a_{2}, \ldots, a_{n}\in T$ such that $a_{i}\in A_{i}$ and
$m(A_{i} \setminus \{a_{i}\})=0,$ for every $i\in \{1,\ldots,n\}.$ Then we
have
\begin{equation*}
0\leq m(T \setminus \{a_{1}, \ldots, a_{n}\}) \leq m(A_{1} \setminus
\{a_{1}\}) +\ldots + m(A_{n} \setminus \{a_{n}\}) = 0,
\end{equation*}
which implies $m(T \setminus \{a_1, \ldots, a_n\})= 0.$ Now, since $m$
is null-additive it follows $m(T) = m(\{a_1,\ldots,a_n \})$. 
\end{proof}

 We finish this section by presenting a property of total
measurability on atoms. Let $\mathcal{A}$ be an algebra of subsets of $T, m:\mathcal{A} \rightarrow
[0,+\infty )$ a non-negative set function with $m(\emptyset )=0$ and $(X,\| \cdot \|)$ a Banach space.

\begin{definition}\label{3.1} \rm
 A \textit{partition} of $T$ is a finite family of nonvoid sets 
$P=\{A_i \}_{i\leq n} \subset \mathcal{A}$ such that 
$A_i \cap A_j = \emptyset , i\neq j$ and
 $\cup_{i \leq n} A_i=T.$ 
 Let $P=\{A_i \}_{i \leq n}$ and $P^{\prime}=\{B_j \}_{j \leq m}$
be two partitions of $T$. 
$P^{\prime} $ is said to be \textit{finer than} $P$, denoted by $P\leq P^{\prime}$
 (or, $P^{\prime}\geq P$), if for every $j \in\{1,\ldots, m \}$, 
there exists $i_j \in \{1,\ldots,n \}$ so that $B_j \subseteq A_{i_j}$.
\\ The \textit{common refinement} of two partitions $
P=\{A_i \}_{i \leq n}$ and $P^{\prime}=\{B_j \}_{j \leq m}$ is the
partition 
$P\vee P^{\prime}=\{A_i \cap B_j \}_{ i\in\{1,\ldots,n\},  j\in\{1,\ldots,m\}}$.
\end{definition}

We denote by $\mathcal{P}$ the class of all partitions of $T$ and if $A\in\mathcal{P}(T)$
 is fixed, by $\mathcal{P}_{A}$ the class of all partitions of $A$.

\begin{definition}\rm 
 A vector function $f:T\rightarrow X$ is said to be: 
\begin{itemize}
\item
 \textit{$m$-totally-measurable} (\textit{on} $T)$ if for every $
\varepsilon >0$ there exists a finite family $\{A_{i}\}_{i=0}^{n}\subset
\mathcal{A}$ of pairwise disjoint sets, with $\{A_{1},\ldots ,A_{n}\}\subset
\mathcal{A}\setminus \{\emptyset \}$, such that the following two
conditions hold:
\begin{equation*}
\begin{cases}
\overline{m}(A_{0})<\varepsilon \;\text{ and} \\
\sup\limits_{t,s\in A_{i}}\|f(t)-f(s)\|=\mbox{osc}(f,A_{i})<\varepsilon ,\;
\text{ for every }i\in \{1,\ldots ,n\}.
\end{cases}
\leqno{(*)}
\end{equation*}
\item 
 ${m}$-\textit{totally-measurable on $B\in \mathcal{A}$} if the
restriction $f|_{B}$ of $f$ to $B$ is $m$-totally measurable on 
$(B,\mathcal{A}_{B},m_{B})$, where $\mathcal{A}_{B}=\{A\cap B; A\in \mathcal{A}\}$ and 
$m_{B}=m|_{\mathcal{A}_{B}}$.
\end{itemize}
\end{definition}

\begin{example} \rm
Every simple function $f:T\rightarrow X, 
f=\sum\limits_{i=1}^{n} a_i  1_{A_i}$
(where $\{A_i \}_{i=1}^{n}$ is a partition of $T$ and $1_{A_i}$ is the
characteristic function of $A_i$), 
 is $m$-totally-measurable. 
\end{example}

\begin{remark}\label{rem-mstella} \rm \mbox{}
\begin{itemize}
\item[\ref{rem-mstella}.1)] In the condition ($\ast $) of  
\cite[Definition 4.2]{g1965} for a
totally-measurable function $f:T\rightarrow \mathbb{R}$, instead
of $\overline{m}$ it is used $m^{\ast }$. For a vector measure 
$m:\mathcal{A} \rightarrow X$, $m^{\ast }$ is defined by:
\begin{equation*}
m^{\ast }(E)=\sup \{\Vert m(A)\Vert ;A\in \mathcal{A},A\subseteq E\},\forall
E\in \mathcal{P}(T).
\end{equation*}
We remark that if $m:\mathcal{A}\rightarrow \lbrack 0,+\infty )$ is a
non-negative set function, then
\begin{equation*}
m^{\ast }(E)=\sup \{m(A);A\in \mathcal{A},A\subseteq E\},\forall E\in
\mathcal{P}(T).
\end{equation*}
Thus, if $A\in \mathcal{A}$, then $m(A)\leq \overline{m}(A)$, which implies
that $m^{\ast }(E)\leq \overline{m}(E)$, for every $E\in \mathcal{P}(T)$. In
consequence, if $f$ is $m$-totally-measurable according to our Definition
3.2, then $f$ is also totally-measurable according to 
 \cite[Definition 4.2 ]{g1965} (which we call $m^{\ast }$-totally-measurable). 
\\
We also observe that if $m$ is finitely additive (a subadditive
monotone measure respectively), then, according to Remark \ref{rem-bar}.vii), we have $
\overline{m}(E) = m^{*}(E)$, for every $E\in \mathcal{P}(T)$ ($E\in\mathcal{A}$ respectively). 
So, in these cases, the two definitions coincide. 
\item[\ref{rem-mstella}.2)]  If $f:T\rightarrow X$ is $m$-totally-measurable on $T$, then $f$
is $m$-totally-measurable on every $A\in \mathcal{A}$. 
\item[\ref{rem-mstella}.3)]  If $m$ is null-additive and monotone, and $A\in \mathcal{A}$ is an atom
 for $m$, then $m_A$ is finitely additive, 
where $m_A$ is the restriction of $m$ to $A\cap \mathcal{A}$, and so $m_A=\overline{m}_A$. 
(This allows to avoid the request that $m$ is of finite variation in the Theorem \ref{ex4.10} and subsequent Corollary).
\item[\ref{rem-mstella}.4)]  If $m$ is null-additive and monotone, and $A\subset T$ is an atom for $m$, then a function $f:T\to X$ is totally
 measurable on $A$ if and only if
$$\inf_{U\in \mathcal{U}}\mbox{osc}(f,U)=0,$$
where $\mathcal{U}$ is the family of all atoms contained in $A$.
\item[\ref{rem-mstella}.5)] 
 Suppose $m$ is subadditive and let  $\{A_{i}\}_{i=1}^{p}
\subset \mathcal{A}$. If $f:T\rightarrow X$ is $m$-totally-measurable on
every $A_{i},i\in \{1,\ldots ,p\}$, then $f$ is $m$-totally-measurable on $
\bigcup\limits_{i=1}^{p}A_{i}$. 
\end{itemize}
\end{remark}

\begin{theorem}\label{th-tm}
 Suppose  $f:T\rightarrow X$ is a vector function 
 and $m:\mathcal{B}\rightarrow \lbrack 0,+\infty )$ is a regular
null-additive monotone measure. If $A\in \mathcal{B}$ is an atom of $m$,
then $f$  is $m$-totally-measurable on  $A$.
\end{theorem}
\begin{proof} By Theorem \ref{cor-at}, there exists a unique point $
a\in A$ so that $m(A\setminus \{a\})=0$.
We observe that the partition $P_{A}=\{A\setminus \{a\},\{a\}\}$ assures
the ${m}$-total measurability of $f$ on  $A$.
\end{proof}

 Observe that an analogous result was given in \cite[Theorem 4.9]{gcmg2009} for compact metric spaces.
By  \ref{rem-mstella}.3) and Theorem \ref{th-tm}, we immediately get:

\begin{corollary}
If $f:T\rightarrow X$ is a vector function and $m:\mathcal{B}\rightarrow
\lbrack 0,+\infty )$ is a finitely purely atomic regular subadditive
monotone measure, then $f$ is ${m}$-totally-measurable on $T$.
\end{corollary}
\section{Gould integrability on atoms}

In this section some properties regarding Gould integrability on finitely
purely atomic monotone measure spaces are established: a Lebesgue type
theorem of convergence and comparative results between Gould integrability
and total measurability.
The Gould integral was defined in \cite{g1965} for real functions with respect to a
finitely additive vector measure taking values in a Banach space. Different
generalizations and topics on Gould integrability were introduced and
studied in \cite{cg2015,g2006,g2008,gc2009,gcmg2009,gic2015,pc2003,pc2003a,pgc2010,s2005}.
In what follows, $m:\mathcal{A}\rightarrow \lbrack 0,\infty )$ is a set
function with $m(T)>0$. For an arbitrary {vector} function $f:T\rightarrow X$,
 $\sigma _{f,m}(P)$ (or, if there is no doubt, $\sigma _{f}(P),\sigma
_{m}(P)$ or $\sigma (P)$) denotes the sum $\sum_{i=1}^{n} f(t_{i})m(A_{i}),$ for every partition of $T$, 
$P=\{A_{i}\}_{i=1}^{n}$
 and every $t_{i}\in A_{i},i\in \{1,\ldots ,n\}$.

\begin{definition}\label{g-int} \rm
 A vector function $f:T\rightarrow X$ is said to be (\textit{Gould}
) $m$-\textit{integrable} (on $T)$ if the net $(\sigma (P))_{P\in (
\mathcal{P},\leq )}$ is convergent in  $X$, where $\mathcal{P}$ is
ordered by the relation $"\leq "$ given in Definition \ref{3.1}.
\\
If $(\sigma (P))_{P\in (\mathcal{P},\leq )}$ is convergent, then its
limit is called \textit{the Gould integral of }$f$\textit{\ on }$T$\textit{\
with respect to }$m$, denoted by $(G) \int_{T}fdm$ (shortly $\int_{T}fdm$).
\\ If $B\in \mathcal{A},f$ is said to be $m$\textit{-integrable on }
$B$ if the restriction $f|_{B}$ of $f$ to $B$ is $m$-integrable on 
$(B,\mathcal{A}_{B},m_{B})$.
\end{definition}

\begin{remark}\label{ex4.2} \rm
By Definition \ref{g-int}, the following statements hold:
\begin{itemize}
\item $f$ is $m$-integrable on $T$ if and only if there exists
$\alpha \in X$ such that for every $\varepsilon >0$, there exists a
partition $P_{\varepsilon }$ of $T$, so that for every other partition of $T$
, $P=\{A_{i}\}_{i=1}^{n},$ with $P\geq P_{\varepsilon }$ and every choice of
points $t_{i}\in A_{i},i\in \{1,\ldots ,n\}$, we have $\| \sigma(P)-\alpha \|<\varepsilon $.
\item Let $B,C\in \mathcal{A}$ satisfy
 $B\cap C=\emptyset$. 
If $f:T\rightarrow X$ is $m$-integrable on $B$ and $C$, then $f$ is $
m$-integrable on $B\cup C$ and $\int_{B\cup C} fdm=\int_{B}fdm+\int_{C}fdm.$
\end{itemize}
\end{remark}

\begin{example} \rm 
\begin{itemize}
\item Let $T$ be a finite set, $\mathcal{A}=\mathcal{P}(T)$, $m:\mathcal{A}\to [0,+\infty)$ 
and $f:T\to \mathbb{R}$ be arbitrary. Then $f$
is Gould $m$-integrable and $\int_{T}fdm = \sum\limits_{t\in T}f(t)m(\{t\})$.
\item If $m:\mathcal{A}\to [0,+\infty)$ is finitely additive and $
f:T\to \mathbb{R}$ is simple, $f= \sum\limits_{i=1}^{n}a_{i}\cdot 1_{A_{i}}$, 
then $f$ is Gould $m$-integrable and $\int_{T}fdm =
\sum\limits_{i=1}^{n}a_{i}\cdot m (A_{i})$. 
\item Let $T = \mathbb{N}$, $p\in T$ be fixed, $\mathcal{A}$ and $m$ defined as in 
Example \ref{es1}.ii) and let $f:\mathbb{N}\to \mathbb{R}$ be
defined for every $x\in T$ by
\begin{equation*}
f(x) =
\begin{cases}
x, & x\in \{0,\ldots,p\} \\
0, & x\geq p+1
\end{cases}
\end{equation*}
As we remarked in Example 
\ref{es1}.ii), $T$ is an atom of $m$. Then $f$ is Gould $m $-integrable on $T$ and $\int_{T}fdm = 0.$ 
\end{itemize}
\end{example}

\begin{theorem}\label{4.4}
Suppose $m:\mathcal{B}\rightarrow \lbrack 0,+\infty )$ is a regular
null-additive monotone measure and let $f:T\rightarrow X$ be any vector function. For every atom $A\in \mathcal{B}$, 
$f$ is $m$-integrable on $A$, and
$\int_{A}fdm=f(a)m(A)$, where $a\in A$ is the single point resulting by
Theorem \ref{cor-at}.
\end{theorem}
\begin{proof} Fix any partition $P$ of $A$, finer than $P_0:=\{\{a\}, A\setminus \{a\}\}$. 
Then $P$ is of the type $P=\{\{a\}, B_1,...,B_n\}$, where $m(B_i)=0$ for all $i=1,...,n$ 
(since $m$ is monotone and $m(A\setminus \{a\})=0$). So
$\sigma_f(P)=f(a)m(\{a\})=f(a)m(A)$ thanks to  Remark
\ref{rem3}. 
Since this quantity is constant for every $P$ finer than $P_0$, this is the announced integral.
\end{proof} 

\medskip
More generally, with a quite similar proof, one has
\begin{corollary}
Suppose $m:\mathcal{B}\rightarrow \lbrack 0,+\infty )$ is a finitely purely
atomic regular null-additive monotone measure, with $T=\bigcup%
\limits_{i=1}^{n}A_{i}$, and $\{A_{i}\}_{i=1}^{n}\subset \mathcal{A}$ are
pairwise disjoint atoms of $m$. Then any vector function $f:T\rightarrow X$ is
$m$-integrable on $T$, and $\int_{T}fdm
=\sum\limits_{i=1}^{n}f(a_{i})m(A_{i})$, where $a_{i}\in A_{i}$ is the
single point resulting by Theorem \ref{cor-at}, for every $i\in \{1,\ldots ,n\}.$
\end{corollary}

Observe also that, when $X= \mathbb{R}^+$ and the measure $m$ is also null continuous
 ($m(\cup_n A_n) = 0$ for every increasing sequence $(A_n)_n \in \mathcal{B}$,
with $m(A_n) = 0$ for every $n$), the Gould integral of a non negative function $f$
 is equal to its Choquet integral thanks to \cite[Corollary 4.9]{lmp}.
This result can also be extended in the following way:
\begin{proposition}
Let $f:T \to \mathbb{R}^+$ be a measurable, bounded function and $m: \mathcal{A} \to [0, \infty[$
 be a monotone, null additive, atomic measure.
If $A \in \mathcal{A}$ is an atom then
\[ (G) \int_A f dm = (C) \int_A f dm.\]
\end{proposition}
\begin{proof}
First of all 
let $t_0= \sup \{ t \geq 0 : m(\{f(x) > t \}) = m(A)$.
Obviously $t_0 < \infty$ since $f$ is bounded. So, $m(\{f > t \}) =m(A)$ for every $t < t_0$ and 
$m(\{f \leq t \}) = m(A)$ for every $t > t_0$.
So it is
 \[ (C) \int_A f dm = \int_0^{\infty} m(\{f > t \}) dt = \int_0^{t_0} m(\{f > t \}) dt = t_0 m(A).\]
We prove now that $f$ is Gould  integrable. Let $\varepsilon > 0$ be fixed and let $t_1 < t_0< t_2$ 
be such that
 $(t_2 - t_1)m(A) \leq \varepsilon$.\\
Let $\Pi_i= \{ \{ x \in A:  f \leq t_i\}, \{x \in A: f > t_i\} \}, i=1,2$, so
\begin{eqnarray*}
m(A) &=& m(\{ x \in A:  f >t_1\}) = m(\{ x \in A:  f \leq t_2\})=
\\ &=& m((\{ x \in A:  f >t_1\} \cap \{ x \in A:  f \leq t_2\}).
\end{eqnarray*}
Let $\Pi$ finer that $\Pi_1 \vee \Pi_2$.
By Lemma \ref{1di3}.2) it is \, 
 $t_1 m(A) \leq \sigma(f, \Pi) = f(\xi) m(A) \leq t_2 m(A)$.
Then
\begin{eqnarray*}
\left|  \sigma(f, \Pi) - t_0 m(A) \right| &\leq& (t_2 - t_1)m(A) \leq \varepsilon
\end{eqnarray*}
and then $f$ is Gould integrable.
\end{proof}

\begin{corollary}
{\rm (Lebesgue type)} Suppose $m:\mathcal{B}\rightarrow \lbrack 0,+\infty )$ is a
regular null-additive monotone measure and let $f$, $f_{n}:T\rightarrow X$
be arbitrary functions. If $A\in \mathcal{B}$ is an atom of $m$, then
$$\lim_{n\rightarrow \infty }\int_{A}f_{n}dm = \int_{A}fdm$$
if and only if
$$\lim_{n\rightarrow \infty} f_{n}(a) = f(a),$$ where $a\in A$ is the
single point resulting from Theorem \ref{cor-at}. 
\end{corollary}

In case $m$ is not regular, the results in Theorems \ref{cor-at} and \ref{4.4} are not valid, in general.
In order to give an example, we recall the concept of {\em filter}.
\begin{definition}\label{filter}\rm
Let $Z$ be any fixed  set.  A family  $\mathcal{U}$ of subsets of $Z$
is called  a { \textit{filter}} in $Z$ if and only if
\begin{description}
\item[\ref{filter}.a)] $\emptyset  \not \in {\mathcal U}$,
\item[\ref{filter}.b)] $A \cap B \in {\mathcal F}$ whenever $A,B \in \mathcal{U}$ and
\item[\ref{filter}.c)] $A \in \mathcal{F}, B \supset A \Rightarrow B\in \mathcal{U}.$
\end{description}
A filter $\mathcal{U}$ is an ultrafilter if for every $A \subset Z$ then $A \in  \mathcal{U}$ or $Z \setminus A \in  \mathcal{U}$.
\end{definition}
Given any filter $\mathcal{U}$ of subsets of $Z$, the {\em dual ideal} of $\mathcal{U}$ is the family of all {\em complements} of elements from $\mathcal{U}$. 
Usually the dual ideal will be denoted by $\mathcal{I}_{ \mathcal{U}}$.
If the dual ideal $\mathcal{I}_{ \mathcal{U}}$ contains all finite subsets of $Z$, we say that  $\mathcal{U}$ is a {\em free} filter. 
\\

 Consider now the following example:
\begin{example} \rm
Let $T=[0,1]$ with the $\sigma$-algebra $\mathcal{A}=\mathcal{P}([0,1])$,
 and choose any free ultrafilter $\mathcal{U}$ in $\mathcal{P}([0,1])$. Then define $m(A)=1$ if $A\in \mathcal{U}$, and 0  
otherwise. Clearly, each element of $\mathcal{U}$ is an atom; but, since the ultrafilter is free, for every point
 $a\in [0,1]$ one has $m(\{a\})=0$, and therefore, for any mapping $f$, the quantity $f(a)m(\{a\})$ is always null.
\end{example}

In general, when $m$ is null-additive and monotone, and $A$ is an atom for $m$, integrability of a mapping 
$f:T\to X$ in the set $A$ is strictly related with its total measurability in $A$.\\

Concerning the relationships between Gould integrability and total
measurability, we recall the following:
\begin{itemize}
\item if $m:\mathcal{A}\rightarrow X$ is a finitely additive vector
measure and $f:T\rightarrow \mathbb{R}$ is bounded, then $f$ is Gould $m$
-integrable if and only if $f$ is $m^{\ast }$-totally-measurable (\cite{g1965});
\item if $m:\mathcal{A}\rightarrow \lbrack 0,+\infty )$ is a submeasure
of finite variation and $f:T\rightarrow \mathbb{R}$ is bounded, then $f$ is
Gould $m$-integrable if and only if $f$ is $m$-totally-measurable
( \cite{gp2007}).
\end{itemize}

An interesting result in the present framework is the following.

\begin{theorem}\label{nuovo47}
If\ $m:\mathcal{A}\to [0,+\infty)$ is finitely additive, and $f:T\to X$ is a bounded $m$-totally measurable function, 
then $f$ is Gould-integrable.
\end{theorem}
\begin{proof}
 Denote by $M$ any positive constant dominating $\|f\|$, and fix arbitrarily $\varepsilon>0$. Then there exists a partition 
$P=\{A_0,A_1,...,A_n\}$ of $T$, such that $m(A_0)< \varepsilon (4M)^{-1}$ and
$$\|f(t_j)-f(\tau_j)\|\leq \frac{\varepsilon}{2m(T)}$$
for all points $t_j,\tau_j\in A_j$, $j=1,...,n$.
Now choose arbitrarily two partitions $P',P''$ finer than $P$, and, without loss of generality, assume that $P''$ is finer than $P'$.
Then
\begin{eqnarray*}
\sigma_f(P')-\sigma_f(P'') &=&\sum_{A\in P'}f(t_A)m(A)-\sum_{B\in P''}f(\tau_B)m(B)= \\ 
&=& \sum_{A\in P'}\sum_{\substack{B\in P'',\\B\subset A}}\big( f(t_A)-f(\tau_B)\big) m(B).
\end{eqnarray*}
Therefore
$$\|\sigma_f(P')-\sigma_f(P'')\|\leq \sum_{A\in P'}\sum_{\substack{B\in P'',\\B\subset A}}\|f(t_A)-f(\tau_B)\|m(B).$$
Now, split the summation above into two parts: 
summands for which the sets $A$ are contained in $A_0$ 
and the remaining summands. Then we have
\begin{eqnarray*}
&& \sum_{\substack{A\in P',\\ A \subseteq A_0}} \, 
\sum_{\substack{B\in P'',\\ B\subset A}}\|f(t_A)-f(\tau_B)\|m(B)\leq 2Mm(A_0)\leq \varepsilon/2,\\
&& 
\sum_{\substack{A\in P',\\ A \not\subseteq A_0}} \, \sum_{\substack{B\in P'',\\ 
B\subset A}}\|f(t_A)-f(\tau_B)\|m(B)\leq \frac{\varepsilon}{2M(T)}\sum_{B\in P''}m(B)\leq \varepsilon/2.
\end{eqnarray*}
Thus
$$\|\sigma_f(P')-\sigma_f(P'')\|\leq \varepsilon.$$
Since $P'$ and $P''$ were chosen arbitrarly finer than $P$, this proves Gould integrability, thanks to completeness of $X$.
\end{proof}
\medskip

In the sequel, we obtain some comparative results between Gould
integrability and total measurability for the case when $m$ has weaker
properties.

\begin{theorem}\label{ex4.7}
Suppose $m:\mathcal{A}\rightarrow \lbrack 0,+\infty )$ is a
null-additive monotone set function which has atoms. If 
$f:T\rightarrow X$  is ${m}$-totally-measurable on an atom $A\in \mathcal{A}$, then $f$ is  $m$-integrable on 
$A.$
\end{theorem}
\begin{proof}
 The proof is  more direct than that given in \cite[Theorem 4.13]{gcmg2009} and therefore it is added here.
 Let $A\in \mathcal{A}$ be an atom of $m$. According
to Lemma 3.2, if $\{A_{i}\}_{i=1}^{n}$ is a partition of $A$, then there
exists only one set, for instance, without any loss of generality, $A_{1},$
so that $m(A_{1})=m(A)>0$ and $m(A_{2})=...=m(A_{n})=0.$
From this it clearly follows that $m$ is finitely additive when restricted to the measurable subsets of $A$. 
Moreover, thanks to  \ref{rem-mstella}.4), $f$ is certainly bounded on some atom $U\subset A$. 
Then, thanks to Theorem \ref{nuovo47}, $f$ is Gould integrable on $U$, and it follows easily also integrability on $A$, 
since $m(A\setminus U)=0.$
\end{proof}

By Theorem \ref{ex4.7}  and Remark \ref{ex4.2}, we get:

\begin{corollary}\label{ex4.8}
Suppose\textit{\ }$m$\textit{\ }is a finitely purely atomic null-additive
monotone measure. If $f$\ is ${m}$-totally-measurable on $T$, then $f$\ is 
$m$-integrable on $T.$
\end{corollary}

\begin{theorem}\label{ex4.10}
Let $m:\mathcal{A}\rightarrow \lbrack 0,+\infty )$
be a null-additive monotone set function
 and $f:T\rightarrow X$  an arbitrary function. If  $f$  is  $m$ 
-integrable on an atom  $A\in \mathcal{A}$  of $m$,  then  $f
$  is $m$-totally-measurable on  $A.$ 
\end{theorem}
\begin{proof}
Since $f$ is $m$-integrable on $A$, for every $\varepsilon >0$,
there is $P_{\varepsilon }\in \mathcal{P}_{A}$,
$P_{\varepsilon}=\{C_1,...,C_n\}$, so that
\begin{equation*}
\left\|\sum\limits_{i=1}^{n}f(t_{i})m(C_{i})-\sum
\limits_{i=1}^{n}f(s_{i})m(C_{i}) \right\|<\varepsilon ,
\end{equation*}
for every $t_{i},s_{i}\in C_{i},i\in \{1,2,...,n\}$. Since $A$ is an atom,
suppose $m(A\setminus C_{1})=0,m(C_{i})=0,i\in \{2,3,...,n\}$. It follows that, by   \ref{rem-mstella}.3), 
 $m$ is finitely additive on $A$ and $m(A\setminus C_{1})=0$ 
and so $f$ is $m$-totally-measurable on $A.$
\end{proof}
\medskip

According to Theorems \ref{ex4.7} and \ref{ex4.10}, the following result is obtained:

\begin{corollary}\label{cinque}
Suppose $m:\mathcal{A}\rightarrow \lbrack 0,+\infty )$ is a
null-additive monotone set function 
and $A\in \mathcal{A}$ is an atom of $m$. Then $f$ is 
$m$-integrable on $A$ if and only if $f$  is  $m$-totally-measurable on $A.$
 Moreover, if this is the case, the integral of $f$ on $A$ is equal to $xm(A)$, where $x$ is the  unique element in $X$ such that 
$$\{ x \} =\bigcap_{U\in \mathcal{U}}\overline{f(U)},$$
where $\mathcal{U}$ is the filter of all atoms $U\subset A$.
\end{corollary}
\begin{proof}
The only fact to prove is the conclusion about the integral. Without loss of generality, we shall assume that $m(A)=1$.
Since $f$ is totally measurable on $A$, by  \ref{rem-mstella}.4), it follows that, for every integer $n$ there exists an 
atom $U_n\subset A$ such that $\mbox{osc}(f,U_n) \leq n^{-1}.$
 Without loss of generality, the atoms $U_n$ can be chosen to be decreasing. 
Therefore, choosing arbitrarily an element $u_n\in U_n$ for every $n$, the sequence $(f(u_n))_n$ is Cauchy in $X$ and 
therefore convergent to some element $x$. 
We shall prove now that $x$ is the integral of $f$ in $A$. Indeed, fix arbitrarily $\varepsilon>0$ and pick any integer $n$ larger than $\frac{1}{\varepsilon}$, and such that $\|x-u_n\|\leq \varepsilon$. Then, consider any partition $P$ of $A$, finer than $(U_n,A\setminus U_n)$: setting $P=(A_1,...,A_k)$, and choosing points $a_i\in A_i$, $i=1,...,k$ one has
$$\sigma(f,P)=f(a_j),$$
where $A_j$ is the unique element of $P$ contained in $U_n$ and  belonging to $\mathcal{U}$. Now, we have
\begin{eqnarray*}
\|f(a_j)-x\| &=& \|f(a_j)-x\|\leq (\|x-f(u_n)\|+\|f(u_n)-f(a_j)\|)\leq
 2\varepsilon
\end{eqnarray*}
since both $u_n$ and $a_j$ belong to $U_n$.
This clearly shows that $x$ is the integral of  $f$ on $A$.
Finally, let us prove that  $x$ is the unique point of $X$ belonging to the intersection of all the sets $\overline{f(U)}$, 
as $U$ runs among the atoms contained in $A$.
Indeed, choosing the sequence $(U_n)$ as above, one clearly has $x\in \overline{f(U_n)}$ for every $n$. 
Since  the diameter of the set $f(U_n)$ coincides with  $\mbox{osc}(f,U_n)$ and $diam(f(U_n))=diam(\overline{f(U_n)})$, there is at most one point in common to all the sets $\overline{f(U_n)}$.
Finally, if there exists an element $U\in \mathcal {U}$ such that $x\notin\overline{f(U)}$,
 the same procedure as above can be repeated, replacing $U_n$ with $U_n\cap U$ 
and choosing points $u_n\in U_n\cap U$, showing that the limit of the sequence $f(u_n)$ is still $x$, contradiction.
\end{proof}

An interesting consequence is that {\em any} bounded real function is integrable in a purely finitely atomic space. 
Indeed, we have
\begin{corollary}\label{sei}
Let  $m:\mathcal{A}\rightarrow \lbrack 0,+\infty )$ be a
null-additive finitely purely atomic monotone set function. Then every bounded measurable mapping
 $f:T\to \mathbb{R}$ is $m$-integrable and $m$-totally measurable.
\end{corollary}
\begin{proof}
Of course, it is enough to prove integrability of $f$ on each atom $A$, so we shall assume that $T$ itself is an atom for $m$. 
Let us denote by $\mathcal{U}$ the filter of those elements $U$ of $\mathcal{A}$ such that $m(U)=m(T)$. Thanks to the previous result, and to the Remark \ref{rem-mstella},4) it is enough to prove that
\begin{eqnarray}\label{limsup}
\sup_{U\in \mathcal{U}}\inf_{t\in U}f(t)=\inf_{U\in \mathcal{U}}\sup_{t\in U}f(t).\end{eqnarray}
Of course, the left-hand term is not greater than the right-hand one. So, by contradiction, let us assume that a real number $u$ exists, such that
\begin{eqnarray}\label{limsup2}
\sup_{U\in \mathcal{U}}\inf_{t\in U}f(t)<u<\inf_{U\in \mathcal{U}}\sup_{t\in U}f(t).\end{eqnarray}
Now, let $B:=\{t\in T: f(t)<u\}$. Since $f$ is measurable, $B$ is too, and then either $m(B)=m(T)$ or $m(B^c)=m(T)$.
 In the first case $B\in \mathcal{U}$ but 
$\sup_{t\in B}f(t)\leq u$, contradicting (\ref{limsup2}). In the second case, $B^c\in \mathcal{U}$ but 
$\inf_{t\in B^c}f(t)\geq u$, thus contradicting (\ref{limsup2}) again. 
Then (\ref{limsup}) is true, and integrability of $f$ is proved. Then, thanks to Corollary \ref{cinque}, $f$ is $m$-totally measurable on every atom $A$, and therefore on the whole of $T$.
\end{proof}

However, as soon as $X$ is infinite-dimensional, there exist $X$-valued bounded {\em measurable} maps on some atomic space $T$ that are not integrable, as the following example shows.

\begin{example}\rm
Let $X$ be any infinite-dimensional Banach space, and denote by $B_X$ its unit ball. Of course, $B_X$ is not compact, so there exist an $\varepsilon>0$ and a  sequence $(x_n)_n$ in $B_X$, such that
$\|x_n-x_m\|\geq \varepsilon$ whenever $n\neq m$.
Denote by $Y$ the countable set $\{x_n:n\in \enne\}$ and let $\varphi:\enne\to Y$ be any bijection. Now, let $T$ be the space $\enne$ endowed with the $\sigma$-algebra $2^T$. Moreover, let $\mathcal{U}$ be any free ultrafilter on $T$, and $m:2^T\to \{0,1\}$ be the ultrafilter measure associated to $\mathcal{U}$, i.e. such that $m(E)=1$ if 
$E\in \mathcal{U}$ and $m(E)=0$ otherwise. Now, the function $\varphi$ above, mapping $T$ into $X$, 
is clearly bounded and measurable (in the sense that the inverse image of any Borel subset of $X$ is in the
 $\sigma$-algebra fixed in $T$), but is not integrable, since the set $\varphi(U)$ has diameter larger than
 $\varepsilon$ for every $U\subset T$ with $m(U)=1$.
\end{example}

The following result states that, under several supplementary conditions, any
sequence of totally-measurable bounded vector functions is uniformly bounded
almost everywhere.
\begin{theorem}
Let $\mathcal{A}$ be a $\sigma $-algebra and $m:\mathcal{A}\rightarrow
\lbrack 0,+\infty )$ be finitely purely atomic and $\sigma $-null-additive.
Suppose that for every $n\in \mathbb{N}$, $f_{n}:T\rightarrow X$ is  $m$-integrable on $T$. 
If there is $K>0$ so that $\|\int_{A}f_{n}dm\|
\leq K$, for every $n\in \mathbb{N}$ and every atom $A\in \mathcal{A}$,
there exists $U\in \mathcal{A}$ so that $m(T\setminus U)=0$ and $
(f_{n})_{n} $ is uniformly bounded on $U$.
\end{theorem}
\begin{proof}
Without loss of generality, we shall assume that $T$ itself is an atom, and  that the integrals $x_n:=\int_T f_n dm$ are contained in the ball $B(0,K)$ of $X$.
Thanks to the previous Corollary \ref{cinque}, for each $n$ there exists an atom $U_n\subset T$ such that 
$diam(f_n(U_n))\leq 1$
and $x_n\in \overline{f_n(U_n)}$; hence $f_n(U_n)\subset B(x_n,1)\subset B(0,K+1)$ for all $n$. 
Consider now $U:=\bigcap_n U_n$. Since $m(U_n^c)=0$ and $m$ is $\sigma$-null-additive, it follows that 
$m(\bigcup_n U_n^c)=0$, and therefore $U$ is an atom. Of course, for every $n$, $f_n(U)\subset f_n(U_n)
\subset B(0,K+1)$ and this concludes the proof.
\end{proof}

\begin{theorem}
Let $\mathcal{A}$ be a $\sigma $-algebra and $m:\mathcal{A}\rightarrow
\lbrack 0,+\infty )$ be finitely purely atomic and $\sigma $-null-additive.
Suppose that for every $n\in \mathbb{N}$, $f_{n}:T\rightarrow X$ is  $m$-integrable on $T$. 
Moreover, assume that, for every atom $A$, there exists in $X$ the limit
$$\lim_n \int_A f_n dm=x(A).$$
Then, for every atom $A$ there exists an atom $U\subset A$ such that $f_n$ uniformly converges in $U$ to the constant $x(A)$.
\end{theorem}
\begin{proof}
Again, without loss of generality, we shall assume that $T$ is an atom for $m$. Let us denote $x_n:=\int_T f_n dm$, and $x:=\lim_n x_n$.
Thanks to the Corollary \ref{cinque}, for every $n$ there exists an atom $U_n$ such that $diam(f_n(U_n))\leq \frac{1}{n}$, and $x_n\in \overline{f_n(U_n)}$. Without loss of generality, we shall assume that $U_n\subset U_{n-1}$ for each $n$. Since $m$ is $\sigma$-null-additive, then $U:=\bigcap U_n$ is still an atom: we shall prove now that the functions $f_n$ uniformly converge on $U$ to the constant $x$.
Fix $t\in U$ and $n$. Since $f_n(t)\subset\overline{f_n(U_n)}$, we get that $\|f_n(t) - x_n\|\leq \frac{1}{n}.$ Then, for every positive $\varepsilon$, an integer $n_0$ exists, such that $\|x_n-x\|\leq \varepsilon$
as soon as $n>n_0$; therefore, when $n\geq n_0$,
$$\|f_n(t)-x\|\leq \varepsilon$$
for all $t\in U$.
\end{proof}

\section{Radon-Nikod\'ym Theorem}
In this section we shall investigate the behavior of the integral measure with respect to a finitely purely 
atomic measure $m$, and shall deduce the existence of a Radon-Nikod\'ym derivative, under mild conditions.

\begin{proposition}\label{necessary}
Let $m:\mathcal{A}\to [0,+\infty)$ be a null-additive, monotone, and finitely purely atomic measure. 
For every Gould integrable mapping $f:T\to X$, denote by $\mu$ the integral measure of $f$, i.e.
$$\mu(B)=\int_B f dm.$$
Then $\mu$ is finitely additive, finitely purely atomic and absolutely continuous with respect to $m$, i.e. 
$m(E)=0$ implies that $\mu(E)=0$ for $E\in \mathcal{A}$.
\end{proposition}
\begin{proof} First of all, thanks to the Remark \ref{ex4.2}, $\mu$ is finitely additive. Since $m$ is monotonic, 
it is obvious that $m(E)=0$ implies that $\int_E f dm=0$, so $\mu$ is absolutely continuous with respect to $m$. 
\\
Now, let $A$ be an atom for $m$, and assume that $A$ is not an atom for $\mu$. If $A$ is not null for $\mu$ then 
there exists a measurable $B\subset A$ such that $\mu(B)\neq 0$ and $\mu(A\setminus B)\neq 0$: but $B$ or
 $A\setminus B$ is null for $m$, and so also for $\mu$, so the only possibility is that $A$ is null for $\mu$.
\\
This shows that each atom $A$ for $m$ is also an atom for $\mu$, unless $\mu(B\cap A)=0$ for all $B\in \mathcal{A}$, 
and this implies that $\mu$ is purely finitely atomic. 
\end{proof}

\begin{proposition}
Let $m$ and $f$ as above, and denote by $A_1,A_2,...,A_k$ the atoms of $m$. For each $i=1,...,k$, let $a_i$ 
denote the integral of $f$ in $A_i$. Then, for every set $B\in \mathcal{A}$,
$$\mu(B)=\sum_{i=1}^k \frac{a_i}{m(A_i)}m(B\cap A_i)$$
\end{proposition}
\begin{proof} Indeed, for every $B\in \mathcal{A}$, one has
$$\mu(B)=\int_B f dm=\sum_{i=1}^k\mu(B\cap A_i).$$
Now, for each index $i$, there are two possibilities: either $A_i$ is an atom for $\mu$, or $a_i=\mu(A_i)=0$. 
Let  $H$ denote the set of indexes $i$ for which $A_i$ is an atom for $\mu$, and $H_B$ denote the subset of $H$ 
of those indexes for which $\mu(A_i\cap B)=\mu(A_i)$. Then clearly
$$\mu(B)=\sum_{i\in H_B}\mu(B\cap A_i)=\sum_{i\in H_B}\mu(A_i)=\sum_{i\in H_B}a_i.$$
Now, if $i\notin H$, we have $a_i=0$ and so
$$\sum_{i=1}^k \frac{a_i}{m(A_i)}m(B\cap A_i)=\sum_{i\in H} \frac{a_i}{m(A_i)}m(B\cap A_i).$$
Furthermore, if $i\in H\setminus H_B$, this means that $\mu(A_i\cap B)=0$ and so $\mu(A_i\setminus B)=\mu(A_i)$,
 hence $m(A_i\setminus B)>0$ and therefore $m(A_i\cap B)=0$. For this reason
$$\sum_{i\in H} \frac{a_i}{m(A_i)}m(B\cap A_i)=\sum_{i\in H_B} \frac{a_i}{m(A_i)}m(B\cap A_i)=
\sum_{i\in H_B} a_i.$$
This concludes the proof.
\end{proof}

Now we can state the following version of the Radon-Nikod\'ym theorem.

\begin{theorem}
Let $m:\mathcal{A}\to [0,+\infty)$ be a null-additive, monotone, and finitely purely atomic measure.
 Let also $\mu:\mathcal{A}\to X$ be any finitely purely atomic finitely additive measure, absolutely
 continuous with respect to $m$. Then there exists a function $f:T\to X$, Gould integrable with respect to $m$, and satisfying
$$\int_B f dm=\mu(B),$$
for all $B\in \mathcal{A}.$
\end{theorem}
\begin{proof}
 Let $A_1,...,A_k$ denote the atoms of $m$. Then each $A_i$ is also an atom for $\mu$, unless $\mu(B\cap A_i)=0$ for all $B\in \mathcal{A}$.
So, define $f:T\to X$ as follows:
$$f=\sum_{i=1}^k \frac{\mu(A_i)}{m(A_i)}1_{A_i}.$$
Of course, $f$ is simple, and therefore integrable. Moreover, for each set $B\in \mathcal{A}$
$$\int_B f dm=\sum_{i=1}^k \frac{\mu(A_i)}{m(A_i)}m(B\cap A_i).$$
As before, let  $H$ denote the set of indexes $i$ for which $A_i$ is an atom for $\mu$, and $H_B$ denote the subset of
 $H$ of those indexes for which $\mu(A_i\cap B)=\mu(A_i)$. 
Then
$\mu(B)=\sum_{i\in H_B}\mu(A_i),$
and
$$\int_B f dm=\sum_{i\in H} \frac{\mu(A_i)}{m(A_i)}m(B\cap A_i)=\sum_{i\in H_B}\frac{\mu(A_i)}{m(A_i)}m(B\cap A_i)=\sum_{i\in H_B}\mu(A_i)=\mu(B)$$
as desired.
\end{proof}

\section*{Acknowledgements}
The first and the last  authors have been supported by University of Perugia -- Department of Mathematics and Computer Sciences-- Grant Nr 2010.011.0403 and, respectively, by Prin "Metodi logici per il trattamento dell'informazione",  "Descartes" and by the Grant prot. U2014/000237 of GNAMPA - INDAM (Italy).


\begin{thebibliography}{}
\bibitem{bcs2015}  Boccuto, A., Candeloro, D., Sambucini, A.R., \textit{Henstock
multivalued integrability in Banach lattices with respect to pointwise non
atomic measures},  
 Atti Accad. Naz. Lincei Rend. Lincei Mat. Appl. {\bf 26}(4),  363-383 (2015) DOI: 10.4171/RLM/710  
arXiv: 1503.08285v1 [math. FA]. 
\bibitem{cs2014}  Candeloro, D., Sambucini, A.R., \textit{Order-type Henstock and
McShane integrals in Banach lattice setting}, Sisy 20014- IEEE 12th International Symposium on Intelligent Systems and Informatics, Subotica - Serbia; 55-59 (2014),  arXiv:1401.7818 [math.FA], DOI: 10.1109/SISY.2014.6923557 
\bibitem{cs2015a}  Candeloro D., Sambucini A.R., \textit{Comparison between some norm and order gauge integrals in Banach lattices}, PanAmerican Mathematical Journal, {\bf 25} (3), 1-16, 2015
\bibitem{ccgs2015}  Candeloro, D., Croitoru, A., Gavrilu\c{t}, A., Sambucini, A.R.,
\textit{An extension of the Birkhoff integrability for multifunctions},
Mediterr. J. Math., (2015), Doi 10.1007/s00009-015-0639-7.
\bibitem{ccgs2015b}  Candeloro D.,  Croitoru  A., Gavrilut  A. and  Sambucini A.R.,  \textit{
Radon Nikodym theorem for a pair of multimeasures with respect to the Gould integrability}, (2015) arXiv:1504.04110 [math.FA]
\bibitem{cdms2016} Candeloro D., Di Piazza L, Musia{\l} K and  Sambucini A.R.,  \textit{
Gauge integrals and selections of weakly compact valued multifunctions}, submitted, arXiv 1602.00473 (2016).
\bibitem{c1975}  Chi\c{t}escu, I., \textit{Finitely purely atomic measures and $L_{p}$-spaces}, An. Univ. Bucure\c{s}ti \c{S}t. Natur. {\bf 24}, 23-29  (1975).
\bibitem{c2001}  Chi\c{t}escu, I.,\textit{Finitely purely atomic measures:
coincidence and rigidity properties}, Rend. Circ. Mat. Palermo {\bf 50}
(3), 455-476 (2001).
\bibitem{cg2015}  Croitoru, A., Gavrilu\c{t}, A., \textit{Comparison between
Birkhoff integral and Gould integral}, Mediterr. J. Math., {\bf 12},
329-347 (2015).
\bibitem{d1974}  Dobrakov, I., \textit{On submeasures, I}, Dissertationes Math., (1974).
\bibitem{d1972}  Drewnowski, L., \textit{Topological rings of sets,
continuous set functions. Integration}, I, II, III, Bull. Acad. Polon. Sci. S\'{e}r. Math. Astron. Phys. {\bf 20},  269-286 (1972).
\bibitem{eh}  Elton, J.H., Hill, P.T., \textit{Ham Sandwich with Mayo: A
stronger conclusion to the Classical Ham Sandwich Theorem},
arXiv.org/pdf/0901.2589.
\bibitem{g2006}  Gavrilu\c{t}, A., \textit{On some properties of the Gould
type integral with respect to a multisubmeasure}, An. \c{S}t. Univ. ''Al. I.
Cuza'' Ia\c{s}i, {\bf  LII} (1), 177-194 (2006).
\bibitem{g2008}  Gavrilu\c{t}, A., \textit{The Gould type integral with
respect to a multisubmeasure}, Math. Slovaca, {\bf 58}, 1-20  (2008).
\bibitem{gc2009}  Gavrilu\c{t}, A.,  Croitoru A.,  \textit{ Non-atomicity for fuzzy and non-fuzzy multivalued set functions},
Fuzzy Sets and Systems, {\bf 160} (14),  2106-2116 (2009).
\bibitem{gcmg2009} Gavrilu\c{t} A., Croitoru A., Mastorakis N.E., Gavrilu\c{t} G., 
\textit{Measurability and Gould integrability in finitely purely atomic multisubmeasure spaces}, WSEAS Trans. on Math.
 {\bf 8}, 435-444 (2009).
\bibitem{gic2015}  Gavrilu\c{t}, A. C., Iosif A. E., Croitoru, A.,\textit{
The Gould integral in Banach lattices}, Positivity {\bf 19}, 65-82  (2015).
\bibitem{gp2007}  Gavrilu\c{t}, A., Petcu, A., \textit{A Gould type integral with respect to a submeasure}, 
An. \c{S}t. Univ. Ia\c{s}i, {\bf LIII} (2), 351-368 (2007).
\bibitem{g1965}  Gould, G.G., \textit{On integration of vector-valued
measures, }Proc. London Math. Soc. {\bf 15}, 193-225 (1965).
\bibitem{klmp2015} Klement  E.P., Li J., Mesiar R.,  Pap E., \textit{Integrals based on monotone set functions},
Fuzzy Sets and Systems, {\bf 281}, 88-102 (2015); DOI:10.1016/j.fss.2015.07.010
\bibitem{is2010}  Ionascu, E.J., Stancu, A.A.,\textit{On independent sets in
purely atomic probability spaces with geometric distribution}, Acta Math.
Univ. Comenianae, {\bf 79} (1), 31-38 (2010). 
\bibitem{l1997}  Leung, D.H., \textit{Purely non-atomic weak $L^{p}$ spaces},
Studia Mathematica, {\bf 122}, 55-66  (1997).
\bibitem{lmp2014}   Li J., Mesiar R.,  Pap E., \textit{Atoms of weakly null-additive monotone measures and
integrals}, Information Sciences, {\bf 257}, 183–192 (2014).
\bibitem{lmp} Li, J.,   Mesiar, R.,   Pap, E., \textit{ Atoms of weakly null-additive monotone measures and
integrals}, Information Sciences {\bf 257}, 183--192  (2014).
\bibitem{ms1994}  Martellotti A., Musial K.,  Sambucini A. R., \textit{ A Radon-Nikodym theorem for the Bartle-Dunford-Schwartz Integral with respect to Finitely Additive Measures}, Atti Sem. Mat. Fis. Univ. Modena {\bf XLII},  625-633 (1994). 
\bibitem{m2010}  Matveev, M., \textit{Rothberger property and purely atomic
measures}, Questions Answers Gen. Topology {\bf 28}, 157--160 (2010).
\bibitem{m2008} Mesiar R., \textit{Fuzzy integrals and linearity}, International Journal of Approximation Reasoning 
{\bf 47}, 352–358  (2008).
\bibitem{nm2004}  Narukawa Y.,  Murofushi T., \textit{Regular null-additive measure and Choquet integral}, 
Fuzzy Sets and Systems {\bf 143}  487–492 (2004).
\bibitem{p1994}  Pap E., \textit{The range of null-additive fuzzy and non-fuzzy measure}, Fuzzy Sets and Systems 
{\bf 65},  105–115 (1994).
\bibitem{p1995}  Pap, E.,  \textit{Null-additive Set Functions}, Kluwer Academic Publishers, Dordrecht-Boston-London, (1995).
\bibitem{p1995a}  Pap, E.,  \textit{Regular null additive monotone set functions}
, Univ. Novom Sadu, Zb. Rad. Privod-Mat. Fak., Ser. Mat., {\bf 25} (2)  93-101, (1995).
\bibitem{pc2003}  Precupanu, A., Croitoru, A.,  \textit{A Gould type integral
with respect to a multimeasure. I}, An. \c{S}t. Univ. ''Al.I. Cuza'' Ia\c{s}i, {\bf 48}, 165-200 (2002).
\bibitem{pc2003a}  Precupanu, A., Croitoru, A., \textit{A Gould type integral
with respect to a multimeasure. II}, An. \c{S}t. Univ. ''Al.I. Cuza'' Ia\c{s}i, {\bf 49}, 183-207 (2003).
\bibitem{pgc2010}  Precupanu, A., Gavrilu\c{t}, A., Croitoru, A., \textit{A fuzzy Gould type integral}, Fuzzy Sets and Systems {\bf 161} (5), 661-680 (2010).
\bibitem{s2005}  Satco, B., \textit{A Vitali type theorem for the set-valued
Gould integral}, An. \c{S}t. Univ. ''Al. I. Cuza'' Ia\c{s}i. Mat., {\bf 51}, 191-200 (2005).
\end{thebibliography}
\end{document}